\documentclass{amsart}

\usepackage{amssymb}
\usepackage{mathrsfs}
\usepackage{amsmath}
\usepackage{amscd}

\usepackage{graphicx}
\usepackage{subcaption}
\usepackage{wrapfig}
\usepackage{caption}

\usepackage{pinlabel}
\usepackage{enumitem}

\newtheorem{theorem}{Theorem}
\newtheorem{prop}[theorem]{Proposition}
\newtheorem{lemma}[theorem]{Lemma}

\newtheorem{definition}[theorem]{Definition}
\newtheorem{rmk}[theorem]{Remark}

\def\diam{\text{diam}}

\begin{document}

\title{Rigidity of the Delaunay triangulations of the plane}

\renewcommand{\theequation}{\arabic{section}.\arabic{subsection}.\arabic{equation}}
\numberwithin{equation}{section}
\numberwithin{theorem}{section}

\author{Song Dai}
\address{Center for Applied Mathematics and KL-AAGDM, Tianjin University, Tianjin, 300072, P.R. China}
\email{song.dai@tju.edu.cn}

\author{Tianqi Wu}
\address{Department of Mathematics, Clark University, 950 Main St, Worcester, MA 01610, USA
}
\email{mike890505@gmail.com}

\begin{abstract}
    We prove a rigidity result for Delaunay triangulations of the plane  under Luo's notion of discrete conformality, extending previous results on hexagonal triangulations. Our result is a discrete analogue of the conformal rigidity of the plane.
    We follow Zhengxu He's analytical approach in his work on the rigidity of disk patterns, and develop a discrete Schwarz lemma and a discrete Liouville theorem.
    As a key ingredient to prove the discrete Schwarz lemma, we establish a  correspondence between the Euclidean discrete conformality and the hyperbolic discrete conformality, for geodesic embeddings of triangulations.    
    Other major tools include conformal modulus, discrete extremal length, and maximum principles in discrete conformal geometry.
\end{abstract}

\maketitle

\tableofcontents

\section{Introduction}
A fundamental property in conformal geometry is that a conformal embedding of the plane $\mathbb R^2$ to itself must be a similar transformation.
In this paper we discretize the plane by geodesic triangulations and prove a similar rigidity result under the notion of discrete conformal change introduced by Luo \cite{luo2004combinatorial}.

Let $T=(V,E,F)$ be a topological triangulation of a surface with or without boundary, where $V$ is the set of vertices, $E$ is the set of edges, and $F$ is the set of faces. Denote $|T|$ as the underlying space of the complex $T$. A \emph{PL (piecewise linear) metric} on $T$ is a function $l:E\rightarrow \mathbb{R}_+$ such that every triangle $\triangle ijk\in F$ could form a Euclidean triangle under the length $l$.
Luo \cite{luo2004combinatorial} introduced the following notion of \emph{discrete conformality}.
\begin{definition}[\cite{luo2004combinatorial}]
\label{def}
Two PL metrics $l,l'$ on $T=(V,E,F)$ are discretely conformal if there exists a function $u:V\rightarrow\mathbb R$ such that for any edge $ij\in E$,
$$
l'_{ij}=e^{\frac{1}{2}(u_i+u_j)}l_{ij}.
$$
In this case, $u$ is called a discrete conformal factor, and we denote $l'=u*l$.
\end{definition}
Given a PL metric $l$ on $T$, let $\theta^i_{jk}$ denote the inner angle at the vertex $i$ in the Euclidean triangle $\triangle ijk$ under the metric $l$.
The PL metric $l$ is called

\begin{enumerate}[label=(\alph*)]

\item \emph{uniformly nondegenerate} if there exists a constant $\epsilon>0$ such that $\theta^i_{jk}\geq\epsilon$ for all $\triangle ijk$ in $T$,

\item \emph{Delaunay} if
$\theta^{k_1}_{ij}+\theta^{k_2}_{ij}\leq\pi$ for any pair of adjacent triangles $\triangle ijk_1$ and $\triangle ijk_2$ in $T$, and

\item \emph{uniformly Delaunay} if there exists a constant $\epsilon>0$ such that $\theta^{k_1}_{ij}+\theta^{k_2}_{ij}\leq\pi-\epsilon$ for any pair of adjacent triangles $\triangle ijk_1$ and $\triangle ijk_2$ in $T$.
\end{enumerate}
\begin{rmk}
The Delaunay condition is equivalent to that for every pair of adjacent triangles $\triangle ijk_1,\triangle ijk_2\in F$, if the Euclidean quadrilateral $(ik_1jk_2)$ is isometrically embedded in $\mathbb{C}$, then $k_2\notin \emph{int}(D_{ijk_1})$, where $\emph{int}(D_{ijk_1})$ is the interior of the circumscribed disk of $\triangle ijk_1$.
\end{rmk}

A map $\phi:|T|\rightarrow \mathbb{C}$ is called a \emph{geodesic embedding} if for every $ij\in E$, $\phi$ maps $ij$ to a segment connecting $\phi(i)$ and $\phi(j)$, and $\phi$ maps $|T|$ homeomorphically to its image. If further $\phi$ is surjective, we call $\phi$ is a \emph{geodesic homeomorphism} or a \emph{geodesic triangulation}.
It is clear that a geodesic embedding $\phi$ gives a PL metric $l(\phi)$, or $l$ for short, by using the Euclidean distance. A geodesic embedding $\phi$ is called \emph{(uniformly) Delaunay} if $l(\phi)$ is (uniformly) Delaunay.
The main result of the paper is the following.
\begin{theorem}
\label{main}
Suppose $\phi:|T|\rightarrow\mathbb C$ is a geodesic homeomorphism and $\phi':|T|\rightarrow\mathbb C$ is a geodesic embedding with the induced PL metric $l,l'$ respectively, such that

(a) $l,l'$ are both uniformly nondegenerate,

(b) $l$ is uniformly Delaunay and $l'$ is Delaunay, and

(c) $l$ is discretely conformal to $l'$, i.e., $l'=u*l$ for some  $u\in\mathbb R^V$.\\
Then $l$ and $l'$ differ by a constant scaling, i.e., $u$ is constant on $V$.
\end{theorem}
Wu-Gu-Sun \cite{wu2015rigidity} first proved Theorem \ref{main} for the special case where $\phi(T)$ is a regular hexagonal triangulation and $\phi'(T)$ satisfies the uniformly acute condition, i.e. all the inner angles are no more than $\frac{\pi}{2}-\epsilon$ for some constant $\epsilon>0$.
Luo-Sun-Wu \cite{luo2022discrete} and Dai-Ge-Ma \cite{dai2022rigidity}
generalized Wu-Gu-Sun's result by allowing $l'$ to be only Delaunay rather than uniformly acute.
All these works essentially rely on the lattice structure of the regular hexagonal triangulation, and apparently cannot be generalized to triangulations without translational invariance.

\begin{rmk}
In \cite{bobenko2015discrete}, Bobenko-Pinkall-Springborn observed a fundamentally important connection between Luo's discrete conformality and the hyperbolic polyhedra in $\mathbb{H}^3$. As a consequence, the rigidity problem for discrete conformality is indeed equivalent to a Cauchy rigidity problem for ideal hyperbolic polyhedra. See \cite{dai2022rigidity} for more details.
\end{rmk}

To prove Theorem \ref{main}, we follow the approach developed by Zhengxu He in his state-of-the-art work on the rigidity of disk patterns \cite{he1999rigidity}. Theorem \ref{main} immediately follows from the following two propositions.
\begin{prop}\label{boundedness}
Under the conditions of Theorem \ref{main}, the discrete conformal factor $u$ is bounded on $V$. Furthermore, the condition (b) could be relaxed to that both $l,l'$ are just Delaunay.
\end{prop}
\begin{prop}\label{boundedness to rigidity}
Under the conditions of Theorem \ref{main}, if the discrete conformal factor $u$ is bounded on $V$, then it is constant on $V$.
\end{prop}

The proof of Proposition \ref{boundedness} relies on a discrete Schwarz lemma, and estimating conformal moduli for annuli.
To obtain the discrete Schwarz lemma, we estabalish a novel connection between the Euclidean discrete conformality and the hyperbolic discrete conformality, for geodesic embeddings of triangulations. Additionally we develop a maximum principle for hyperbolic discrete conformality.

The proof of Proposition \ref{boundedness to rigidity} is by constructing a discretely conformal geometric flow from $l$ toward $l'$, keeping the surface flat. The derivative of the conformal factor in this flow is known to be discrete harmonic, and thus constant by a discrete Liouville theorem.

\subsection{Notations and conventions}\label{Notations and conventions}
Given $0<r<r'$, denote $D_r=\{z\in\mathbb C:|z|<r\}$
and
$A_{r,r'}=\{z\in\mathbb C:r<|z|<r'\}$.
We also denote $D=D_1$ as the unit open disk.
Given a subset $X$ of $\mathbb C$,
$X^c$ denotes the complement $\mathbb C\backslash X$ and $\partial X$ denotes the boundary of $X$ in $\mathbb C$ and $\text{int}(X)$ denotes the interior of $X$ in $\mathbb C$ and
$$
\diam(X)=\sup\{|z-z'|:z,z'\in X\},
$$
denotes the diameter of $X$.
Given two subsets $X,Y$ of $\mathbb C$,
the distance between $X,Y$ is denoted by
$$
d(X,Y)=\inf\{|z-z'|:z\in X,z'\in Y\}.
$$

Given $i\in V$, denote $\deg(i)$ as the number of neighbors of $i$ and $N_i$ as set of neighbors of $i$, i.e., $N_i=\{j\in V:ij\in E\}$. Furthermore, we denote $R_i$ as the union of the triangles in $T$ containing $i$.
$R_i$ is always viewed as the underlying space of the subcomplex generated by the triangles containing $i$.
Such $R_i$ is called a \emph{1-ring neighborhood} of $i$ if $R_i$ is homeomorphic to a closed disk with vertex $i$ mapped to the center of the disk.
Given a subset $V_0$ of $V$, a vertex $i\in V_0\subseteq V$ is called an \emph{inner point} of $V_0$, if $N_i\subseteq V_0$ and $R_i$ is a 1-ring neighborhood of $i$. We denote $\text{int}(V_0)$ as the set of inner points of $V_0$, and $\partial V_0=V_0-\text{int}(V_0)$. In particular, $\partial V$ is the set of vertices of $T$ that are on the boundary of the surface $|T|$.

Given $l\in\mathbb R^{E}$ and $u\in\mathbb R^{V}$, if $u*l$ is a PL metric then
\begin{enumerate}[label=(\alph*)]

\item $\theta^i_{jk}(u)=\theta^i_{jk}(u,l)$ denotes
the inner angle of $\triangle ijk$ at $i$ under $u*l$, and

\item $K_i(u)=K_i(u,l)$ denotes the discrete curvature at $i$ for $i\in \text{int}(V)$
$$
K_i(u)=2\pi-\sum_{jk:\triangle ijk\in F}\theta^i_{jk}(u).
$$
\end{enumerate}
\subsection{Acknowledgement}
The authors would like to thank Huabin Ge for the encouragement. The first author is supported by NSF of China (No.11871283, No.11971244 and No.12071338). The second author is supported by NSF 1760471.

\section{Hyperbolic discrete conformality and maximum principles}
In this section, we relate a Euclidean geodesic embedding of a triangulation with a hyperbolic geodesic embedding using the Poincar\'e disk model. Our key observation is Lemma \ref{Ehconf}, which shows that the Euclidean discrete conformal change is related with the hyperbolic discrete conformal change by a simple formula.
With the help of this connection, we extend a maximum principle for the Euclidean discrete conformality to the case of the hyperbolic discrete conformality. 

\subsection{Hyperbolic geodesic embeddings}
Denote $\mathbb D$ as the 2-dim hyperbolic space, represented as the Poincar\'e disk model.
A map $\phi^h:|T|\rightarrow \mathbb{D}$ is called a \emph{hyperbolic geodesic embedding} if for every $ij\in E$, $\phi^h$ maps $ij$ to a hyperbolic geodesic segment connecting $\phi^h(i)$ and $\phi^h(j)$, and $\phi^h$ maps $|T|$ homeomorphically to its image. A hyperbolic geodesic embedding $\phi^h$ is called \emph{Delaunay} if for any pair of adjacent triangles $\triangle ijk_1,\triangle ijk_2$ in $T$, $\phi^h(k_2)\notin \text{int}(D_{ijk_1})$ where $D_{ijk_1}$ is the circumscribed disk of $\phi^h(\triangle ijk_1)$. Notice that in the Poincar\'e disk model, a hyperbolic disk in $\mathbb D$ is also a Euclidean disk in $\mathbb C$.

Given $i\in \text{int}(V)$ and a Euclidean geodesic embedding $\phi:R_i\rightarrow D$, we say that $\phi$ \emph{induces} the hyperbolic geodesic embedding $\phi^h:R_i\rightarrow\mathbb D$ if $\phi(j)=\phi^h(j)$ for all $j\in \{i\}\cup N_i$. Such $\phi^h$ exists if $l(\phi)$ is non-degenerate and small, and $\phi(R_i)$ is away from $\partial D$.
\begin{lemma}\label{hyperbolic embedding}
Let $\phi:R_i\rightarrow D$ be a Euclidean geodesic embedding where all the inner angles  are at least $\epsilon>0$ under the PL metric $l(\phi)$.
Suppose
\begin{equation}
\label{assumptions for getting hyperbolic embedding}
l_{ij}\leq(1-|\phi(i)|^2)\sin\epsilon,\quad\text{ for every }j\in N_i.
\end{equation}
Then there exists a hyperbolic geodesic embedding $\phi^h:R_i\rightarrow\mathbb D$ such that $\phi^h$ coincides with $\phi$ on the set of vertices, i.e., $\phi(j)=\phi^h(j)$ for any $j\in N_i\cup\{i\}$.
\end{lemma}
\begin{proof}
Denote $m=\deg(i)$ and $z_0=\phi(i)$. Let $z_1,...,z_m$ be the points in $\phi(N_i)$ such that $z_1-z_0,...,z_m-z_0$ are
counterclockwise.
Let $\exp_{z_0}$ be the exponential map with respect to the hyperbolic metric at $z_0$. Identifying $T_{z_0}\mathbb{D}$ with $\mathbb{C}$ by translating $z_0$ to the origin, denote $v(z)=\exp_{z_0}^{-1}z\in \mathbb{C}$ for any $z\in\mathbb{D}$.
Then we only need to show the following claims.
\begin{equation}\label{1}
\arg (\frac{v(z_{k+1})}{v(z_k)})\in(0,\pi),\quad k=1,\cdots,m,
\end{equation}
and
\begin{equation}\label{2}
\sum\limits_{k=1}^{m}\arg\big(\frac{v(z_{k+1})}{v(z_k)}\big)=2\pi,
\end{equation}
where $z_{m+1}=z_1$.

We first show the claim (\ref{1}).
Fix $k\in \{1,\cdots,m\}$, denote
$$
P=\{z\in\mathbb{C}:\arg(\frac{z-z_0}{z_k-z_0})\in(0,\pi)\},
$$
and
$$
P_h=\{z\in\mathbb{D}:\arg(\frac{v(z)}{v(z_k)})\in(0,\pi)\}.
$$
See Figure \ref{Equivalence between Delaunay and convexity} for illustrations.
\begin{figure}[ht]
	 	\centering
	 \begin{subfigure}[ht]{0.4\textwidth}
	 	 	\centering
	 	\includegraphics[width=1\textwidth]{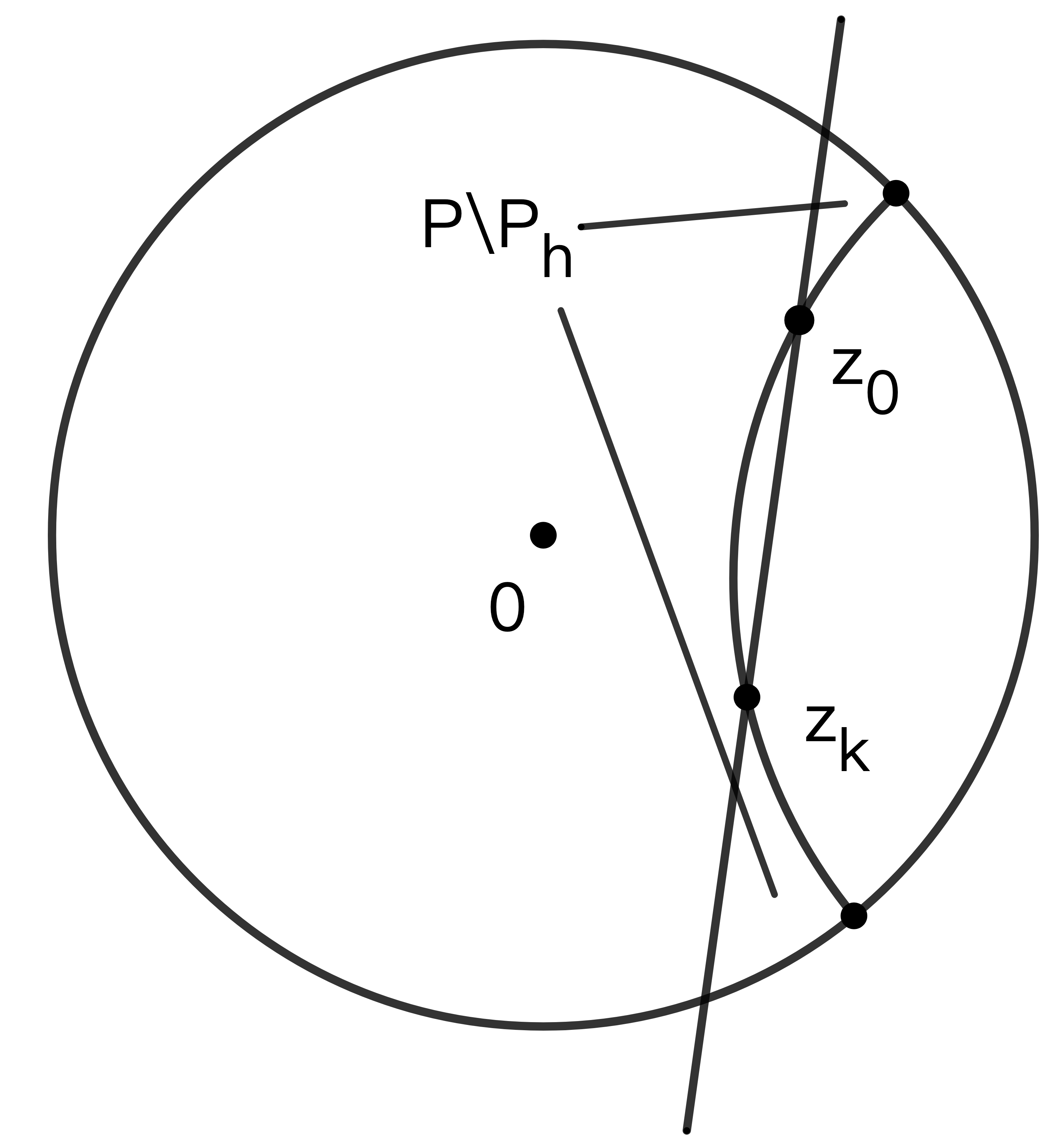}
	 	    \caption{Case 1}
	 \end{subfigure}
 \hspace{0.1cm}
    \begin{subfigure}[ht]{0.4\textwidth}
	\centering
	\includegraphics[width=1\textwidth]{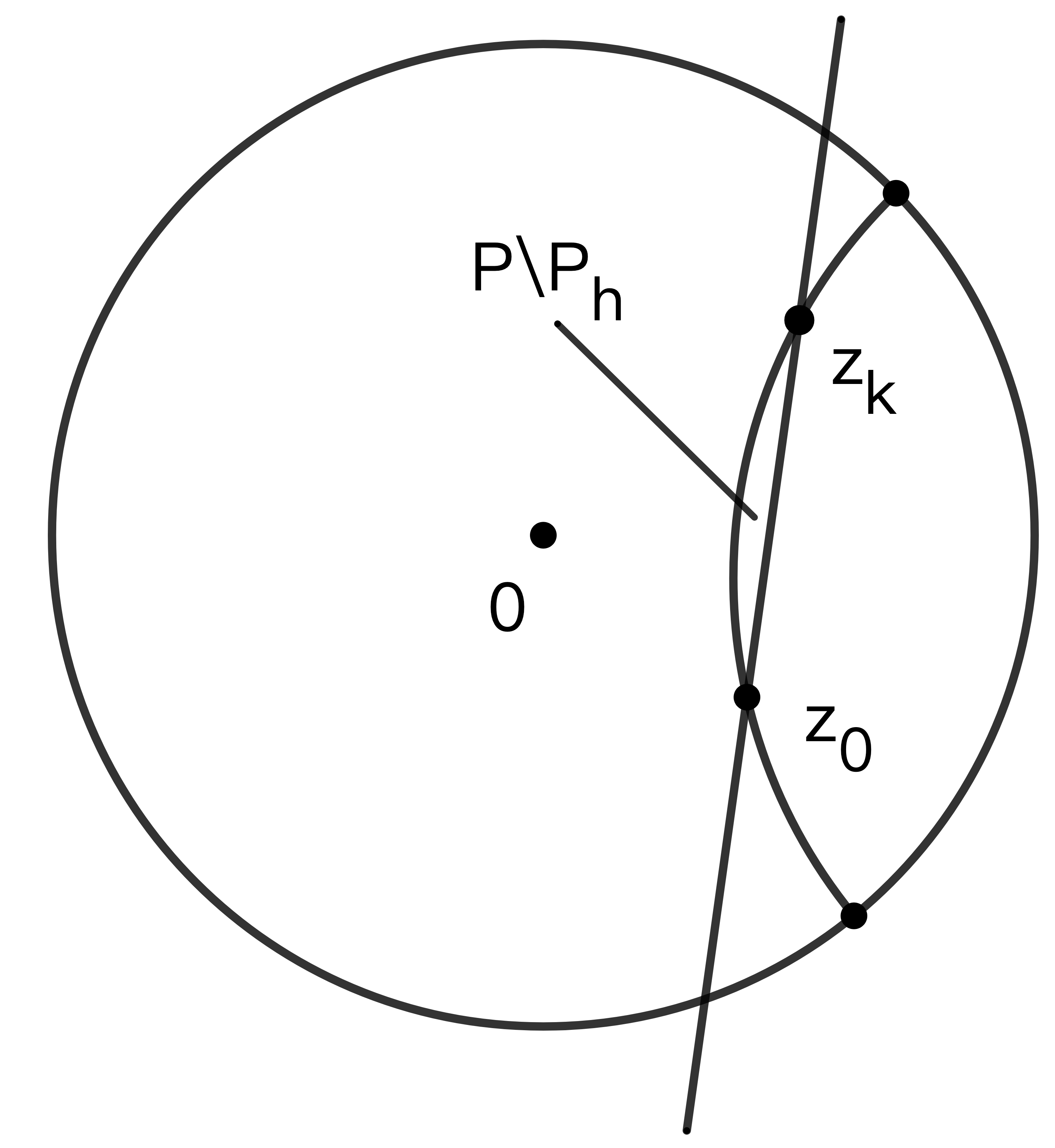}
    \caption{Case 2}
    \end{subfigure}
    \caption{}
\label{Equivalence between Delaunay and convexity}
\end{figure}
Since $\phi$ is a geodesic embedding and $z_1-z_0,...,z_m-z_0$ are counterclockwise,
we have $z_{k+1}\in P$. We need to show $z_{k+1}\in P_h$.
Let $\gamma_h$ be the entire geodesic connecting $z_0$ and $z_k$ with respect to the hyperbolic metric. If $\gamma_h$ is a straight line, then $z_{k+1}\in P_h=P$ and we are done.
Otherwise, $\gamma_h$ is a Euclidean circular arc and orthogonal to the boundary of the unit disk $D$. We denote $z_*$ and $R$ as the Euclidean center and the radius, respectively, of this circular arc. Then
$$
R^2+1=|z_*|^2\leq (|z_0|+R)^2.
$$
So
$$
1-|z_0|^2\leq 2R|z_0|< 2R.
$$
Then by Equation (\ref{assumptions for getting hyperbolic embedding})
$$
\sin\frac{\angle z_0z_*z_{k}}{2}=\frac{|z_k-z_0|}{2R}\leq\frac{(1-|z_0|^2)\sin \epsilon}{2R}< \sin\epsilon.
$$
So
$$
\angle z_0z_{k+1}z_k\geq \epsilon>\frac{1}{2}\angle z_0z_{*}z_{k},
$$
and
$$
\angle z_0z_{k+1}z_k\leq \pi-\epsilon<\pi-\frac{1}{2}\angle z_0z_{*}z_{k}.
$$
Then from the knowledge of the plane geometry, we see $z_k\in P_h$. We finish the proof of claim (\ref{1}).

Next we show the claim (\ref{2}). We have
\begin{equation}\label{3}
\arg\big(\frac{v(z_{k+1})}{v(z_k)}\big)+\arg\big(\frac{v(z_{k})}{z_k-z_0}\big)=
\arg\big(\frac{z_{k+1}-z_0}{z_k-z_0}\big)+\arg\big(\frac{v(z_{k+1})}{z_{k+1}-z_0}\big)+2n\pi,
\end{equation}
for some integer $n$.
From the knowledge of plane geometry, we have that
$$\arg\big(\frac{v(z_{k})}{z_k-z_0}\big)\in(-\frac{\pi}{2},\frac{\pi}{2}).$$
Then together with the claim (\ref{1}), both the right hand side and the left hand side of Equation (\ref{3}) lies in $(-\frac{\pi}{2},\frac{3\pi}{2})$, so $n=0$. Then we obtain
$$\sum\limits_{k=1}^{m}\arg\big(\frac{v(z_{k+1})}{v(z_k)}\big)=\sum\limits_{k=1}^{m}\arg\big(\frac{z_{k+1}-z_0}{z_k-z_0}\big)=2\pi.$$
We finish the proof.
\end{proof}
\begin{rmk}
One can also show that such hyperbolic geodesic embedding $\phi^h$ exists, if the  Euclidean geodesic embedding $\phi$ maps each triangle in $R_i$ to an acute triangle.
\end{rmk}
\begin{rmk}
One might consider using the Klein disk model rather than the Poincar\'e disk model, so that we can immediately have a hyperbolic geodesic embedding from a Euclidean geodesic embedding in $D$, without any further assumption.
However, we have to use the Poincar\'e disk model to establish a key correspondence between the Euclidean discrete conformality and the hyperbolic discrete conformality in Lemma \ref{Ehconf}. It will be clearer in the proof of Lemma \ref{Ehconf}.
\end{rmk}
\subsection{Hyperbolic discrete conformality}\label{Hyperbolic Discrete Conformality}
A \emph{piecewise hyperbolic (PH) metric} on $T$ is also represented by a function $l^h:E\rightarrow \mathbb{R}_+$ such that every $\triangle ijk\in F$ could form a hyperbolic triangle under the length $l^h$.
A hyperbolic geodesic embedding $\phi^h$ naturally gives a PH metric $l^h(\phi^h)$, or $l^h$ for short, by using the hyperbolic distance. 

The notion of the hyperbolic discrete conformality was first introduced by Bobenko-Pinkall-Springborn in \cite{bobenko2015discrete}. 
Let $l^h,l^{h\prime}$ be two PH metrics on $T$. We say $l^h$ is \emph{hyperbolic discretely conformal} to $l^{h\prime}$ if there exists a function $u^h: V\rightarrow \mathbb{R}$ such that for any edge $ij\in E$
$$
\sinh \frac{l^h_{ij}}{2}=e^{\frac{1}{2}(u^h_i+u^h_j)}\sinh \frac{l^{h\prime}_{ij}}{2}.
$$
In this case, $u$ is called a \emph{hyperbolic discrete conformal factor} and we
denote $l^{h\prime}=u^h*^hl^h$.

\begin{lemma}\label{Ehconf}
Let $\phi,\phi^{\prime}:|T|\rightarrow D$ be two geodesic embeddings with the induced PL metrics $l,l^{\prime}$ respectively. Suppose both $\phi,\phi^{\prime}$ induce hyperbolic geodesic embeddings $\phi^h,\phi^{h\prime}:|T|\rightarrow\mathbb{D}$ with the induced hyperbolic PH metrics $l^h,l^{h\prime}$ respectively. Then $l^{\prime}=u*l$ if and only if $l^{h\prime}=u^h*^hl^h$ where $u$ and $u^h$ are related by $$u^h_i=u_i+\ln\frac{1-|z_i|^2}{1-|z_i^{\prime}|^2},$$
with
$z_i=\phi(i)$, $z_i^{\prime}=\phi^{\prime}(i)$.
\end{lemma}
\begin{proof}
Denote $d_h$ as the hyperbolic distance function on $\mathbb{D}$. 
Then the lemma follows from the formula
$$
\sinh\frac{d_h(z_1,z_2)}{2}=
\frac{|z_1-z_2|}
{\sqrt {(1-|z_1|^2)(1-|z_2|^2)} },
$$
where $z_1,z_2\in \mathbb{D}$. See \cite{anderson2005hyperbolic}
 for a proof of the above distance formula.
\end{proof}

\begin{rmk}
In the smooth setting, Lemma \ref{Ehconf} is interpreted as follows. Let $\Omega$ be a domain in $D$ and $f$ be a smooth map from $\Omega$ to $D$, $w=f(z)$. Denote $g_0=|dz|^2$ as the Euclidean metric on $D$ and $g_{-1}=\frac{4}{(1-|z|^2)^2}|dz|^2$ as the hyperbolic metric on $D$. Suppose $f$ is conformal with respect to $g_0$, i.e. $f^*g_0=e^{2u}g_0$ for some smooth function $u=u(z)$. Then $f$ is also conformal with respect to $g_{-1}$. In fact \begin{eqnarray*}
&&\quad f^*g_{-1}=f^*(\frac{4}{(1-|w|^2)^2}|dw|^2)=\frac{4}{(1-|f(z)|^2)^2}e^{2u}|dz|^2\\
&&=\big(\frac{1-|z|^2}{1-|f(z)|^2}\big)^2e^{2u}g_{-1}=e^{2(u+\ln \frac{1-|z|^2}{1-|f(z)|^2})}g_{-1}.
\end{eqnarray*}
\end{rmk}

\subsection{Maximum principles for discrete conformal factors}\label{Maximum Principles for the Discrete Curvature}
Maximum principle plays a very important role in partial differential equations and geometric analysis. For the discrete conformal geometry, the curvature $K$, which is clearly nonlinear as an operator on the set of conformal factors, also satisfies the (strong) maximum principle.

The following lemma is a corollary of Theorem 3.1 in \cite{luo2022discrete}. In \cite{dai2022rigidity} there is another proof of the maximum principle for a special case. Let $T=(V,E,F)$ be a triangulated surface.
\begin{lemma}\label{mp}
Suppose $i\in V$ and $R_i$ is a $1$-ring neighborhood. Let $\phi,\phi^{\prime}:R_i\rightarrow \mathbb{C}$ be two Delaunay geodesic embeddings with induced PL metric $l,l^{\prime}$ respectively. If $l^{\prime}=u*l$, then
$$
u_i\leq \max\limits_{j\in N_i}u_j,\qquad u_i\geq \min\limits_{j\in N_i}u_j.
$$
Furthermore, if $u_i=\max\limits_{j\in N_i}u_j$ or $u_i=\min\limits_{j\in N_i}u_j$, then $u$ is a constant on $\{i\}\cup N_i$.
\end{lemma}
For the hyperbolic setting, we have the following modified version of the maximum principle.
\begin{lemma}\label{hyperbolic maximum principle}
Suppose $i\in V$ and $R_i$ is a $1$-ring neighborhood. Let $\phi^h,\phi^{h\prime}:R_i\rightarrow \mathbb{D}$ be two Delaunay hyperbolic geodesic embeddings with the induced PH metric $l^h,l^{h\prime}$ respectively. Suppose  $l^{h\prime}=u^h*^hl^h$ and then
$u_i^h<0$ implies that
$$
u^h_i> \min\limits_{j\in N_i}u^h_j.
$$
\end{lemma}
\begin{proof}
Since the hyperbolic discrete conformality and conformal factor are invariant under hyperbolic isometries, we may assume $\phi^h(i)=0$, $\phi^{h\prime}(i)=0$. Then $\phi^h,\phi^{h\prime}$ are induced by some Euclidean geodesic embeddings $\phi,\phi'$ with induced PL metrics $l,l^{\prime}$ respectively. Notice that the hyperbolic isometries preserve circles, so $l,l^{\prime}$ are also Delaunay.
From Lemma \ref{Ehconf}, since $l^{h\prime}=u^h*^hl^h$, we have $l$ and $l^{\prime}$ are also discretely conformal. Specifically $l^{\prime}=u*l$ where
$$
u_j=u^h_j-\ln\frac{1-|z_j|^2}{1-|z_j^{\prime}|^2}.
$$
In particular $u_i=u^h_i$. From the maximum principle Lemma \ref{mp}, there exists $j_0\in N_i$ such that $u_{j_0}\leq u_i$. Suppose $u^h_i< 0$. Then
$$
|z^{\prime}_{j_0}|=l^{\prime}_{ij_0}=e^{\frac{1}{2}(u_i+u_{j_0})}l_{ij_0}< l_{ij_0}=|z_{j_0}|.
$$
Therefore
$$
u_{j_0}^h=u_{j_0}+\ln\frac{1-|z_{j_0}|^2}{1-|z_{j_0}^{\prime}|^2}\leq u_i+\ln\frac{1-|z_{j_0}|^2}{1-|z_{j_0}^{\prime}|^2}< u_i=u_i^h.
$$
So
$$
u^h_i> \min\limits_{j\in N_i}u^h_j.
$$
\end{proof}
\begin{rmk}
In the smooth setting, the maximum principles above are interpreted as follows. Let $\Omega$ be a domain in $\mathbb{C}$. Let $f$ be a conformal map from $\Omega$ to $\mathbb{C}$ with respect to the Euclidean metric $g_0=|dz|^2$. Suppose $f^*g_0=e^{2u}g_0$. From the curvature formula, for $g=g(z)|dz|^2$,
$$
K_g=-\frac{2}{g(z)}\partial_z\partial_{\bar{z}}\ln g(z),
$$
we have $\triangle u=0$. So $u$ satisfies the maximum principle. 

On the other hand, let $\Omega$ be a domain in $D$. Let $f$ be a conformal map from $\Omega$ to $D$ with respect to the hyperbolic metric $g_{-1}=\frac{4}{(1-|z|^2)^2}|dz|^2$. Suppose $f^*g_{-1}=e^{2u^h}g_{-1}$. From the curvature formula, we have
$$
\triangle_{g_{-1}}u^h=e^{2u^h}-1,$$
where $\triangle_{g_{-1}}=(1-|z|^2)^2\partial_z\partial_{\bar{z}}$. So $u^h$ satisfies the maximum principle.
\end{rmk}

\section{Preparations for the Proof of the main theorem}
In this section we review properties on discrete Laplacian, conformal modulus of annuli, and discrete extremal lengths. These properties are known to experts and necessary for our proof of Theorem \ref{main}. 
\subsection{Discrete Laplacian on graphs}\label{Discrete Harmonic Functions}
Let $G=(V,E)$ be a connected simple graph, and $\mu:E\rightarrow[0,+\infty)$ be a function on the set of edges. We call $G_{\mu}=(V,E,\mu)$ a \emph{weighted graph}, or an \emph{electrical network}.
The \emph{discrete Laplacian operator} $\Delta_\mu:\mathbb{R}^{V}\rightarrow \mathbb{R}^{V}$ is defined as
$$
(\Delta_\mu u)_i=\sum\limits_{j:ij\in E}\mu_{ij}(u_j-u_i).
$$
We say that $u$ is \emph{harmonic} at $i\in V$, if $(\Delta_\mu u)_i=0$. In this case we have the average property
$$
u_i=\sum\limits_{j:ij\in E}\frac{\mu_{ij}}{\sum\limits_{k:ik\in E}\mu_{ik}}u_j
$$
if $\mu_{jk}>0$ for all $jk\in E$. So we have the maximum principle.
\begin{lemma}[Maximum principle for discrete harmonic functions]\label{discrete Laplacian maximum principle}
Suppose $V$ is finite, $\mu_{ij}>0$ for all $ij\in E$, $V_0$ is a proper subset of $V$. If $u$ is harmonic at every point in $V_0$, then $u$ achieves its maximum and minimum on $V-V_0$.
\end{lemma}
We also have the following well-posedness result of the discrete Laplace equation with the Dirichlet boundary condition.
\begin{lemma}\label{Dirichlet}Suppose $V$ is finite, $\mu_{ij}>0$ for all $ij\in E$, $V_0$ is a proper subset of $V$, and $f$ is a given function on $V-V_0$. Then the following equation of $u\in\mathbb R^V$
$$
\Delta_\mu u=0\text{ in $V_0$},\qquad u=f\text{ on $(V-V_0)$}
$$
has a unique solution. Furthermore, the map $(\mu,f)\mapsto u$ is smooth.
\end{lemma}
Lemma \ref{Dirichlet} is well-known. Solving $u\in\mathbb R^{V}$ here is indeed solving a diagonal dominant linear system.

Let $T=(V,E,F)$ be a triangulation and $l$ be a PL metric. Recall that $K_i(u)=2\pi-\sum\limits_{jk:\triangle ijk\in F}\theta^i_{jk}$ is the curvature at $i\in V$ under the metric $u*l$. Then $u\mapsto K(u)$ is smooth map.
From a direct calculation or \cite{luo2004combinatorial}, one have
\begin{equation}\label{deformation of curvature}
dK_i=-\sum\limits_{j:ij\in E}\mu_{ij}(du_j-du_i),
\end{equation}
where
\begin{equation}\label{weight mu}
\mu_{ij}=\mu_{ij}(u)=\frac{1}{2}(\cot \theta^{k_1}_{ij}(u)+\cot \theta^{k_2}_{ij}(u))
\end{equation}
for adjacent triangles $\triangle ijk_1,\triangle ijk_2\in F$.
It is not difficult to check that if $u*l$ is uniformly Delaunay then $\mu_{ij}\geq \epsilon$ for some constant $\epsilon=\epsilon(T,l,u)>0$.
From Equation (\ref{deformation of curvature}), we see that the linearization of $-K$ is in fact a discrete Laplacian operator with respect to $G_{\mu}=(V,E,\mu)$.

\subsection{Modulus of annuli}\label{Modulus of Annuli and Quasiconformal Maps}
We briefly review the notion of conformal modulus. The definitions and properties discussed here are mostly well-known. One may refer \cite{ahlfors2010conformal} and \cite{lehto1973quasiconformal} for more comprehensive introductions.

A \emph{closed annulus} is a subset of $\mathbb C$ that is homeomorphic to $\{z\in\mathbb C:1\leq|z|\leq2\}$.
An \emph{(open) annulus} is the interior of a closed annulus.
Given an annulus $A$,
denote $\Gamma=\Gamma(A)$ as the set of smooth simple closed curves in $A$ separating the two boundary components of $A$. A real-valued Borel measurable function $f$ on $A$ is called \emph{admissible} if
$
\int_\gamma fds\geq1
$
for all $\gamma\in\Gamma$.
Here $ds$ denotes the element of arc length.
The \emph{(conformal) modulus} of $A$ is defined as
$$
\text{Mod}(A)=\inf\{\| f\|^2_2:f\text{ is admissible}\},
$$
where $\|f\|^2_2$ denotes the integral of $|f(z)|^2$ against the 2-dim Lebesgue measure on $A$.
From the definition it is straightforward to verify that $\text{Mod}(A)$ is conformally invariant. Furthermore, if $f:A\rightarrow A'$ is a K-quasiconformal homeomorphism between two annuli, then
$$
\frac{1}{K}\cdot \text{Mod}(A)\leq\text{Mod}(A')
\leq{K}\cdot\text{Mod}(A).
$$
Given $0<r<r'$, denote $A_{r,r'}$ as the annulus $\{z\in\mathbb C:r<|z|<r'\}$.
It is well-known that
$$
\text{Mod}(A_{r,r'})=\frac{1}{2\pi}\ln\frac{r'}{r}.
$$
Intuitively, the modulus measures the relative thickness of an annulus.
If an annulus $A$ in $\mathbb C\backslash\{0\}$ contains $A_{r,r'}$, then it is ``thicker" than $A_{r,r'}$ and we have the monotonicity
$$
\text{Mod}(A)\geq\text{Mod}(A_{r,r'})
=\frac{1}{2\pi}\ln\frac{r'}{r}.
$$
On the other hand, we have that
\begin{lemma}
\label{modulus}
Suppose $A\subseteq\mathbb C\backslash\{0\}$ is an annulus separating $0$ from the infinity. If $\text{Mod}(A)\geq100$, then
$A\supseteq A_{r,2r}$ for some $r>0$.
\end{lemma}

\begin{proof}
Deonte $B$ as the bounded component of $\mathbb C-A$, and $r=\max\{|z|:z\in B\}$ and
$r'=\min\{|z|:z\in\mathbb (B\cup A)^c\}$.
If $r'\geq2r$ we are done. So we may assume $r'<2r$.

Then $D_{2r}\cap\gamma\neq\emptyset$ for all $\gamma\in\Gamma(A)$.
 Let $f$ be the function on $A$ such that $f(z)=\frac{1}{r}$ on $A\cap D_{3r}$ and $f(z)=0$ on $A\backslash D_{3r}$. If $\gamma\in\Gamma$ and
$
\gamma\subseteq D_{3r}$,
$$
\int_\gamma fds= s(\gamma)\cdot\frac{1}{r}
\geq2\cdot\text{diam}(B)\cdot\frac{1}{r}
\geq 2r\cdot\frac{1}{r}>1.
$$
If $\gamma\in\Gamma$ and $\gamma\not\subseteq D_{3r}$, then $\gamma$ is a connected curve connecting $D_{2r}$ and $D_{3r}^c$ and
$$
\int_\gamma fds\geq d(D_{2r}, D_{3r}^c)\cdot\frac{1}{r}=r\cdot\frac{1}{r}=1.
$$
So $f$ is admissible and
$$
\text{Mod}(A)\leq \int_A f^2=\frac{1}{r^2}\cdot\text{Area}(A\cap D_{3r})\leq\frac{\pi(3r)^2}{r^2}=9\pi<100.
$$
This contradicts with our assumption.
\end{proof}

\begin{rmk}
To some extend, Lemma \ref{modulus} is a consequence of Teichm\"uller's result on extremal annuli (see Theorem 4-7 in \cite{ahlfors2010conformal}). The constant $100$ is chosen for convenience and should not be optimal.
\end{rmk}
\subsection{Recurrence and the discrete extremal length}\label{Recurrence of Electrical Networks}
The Liouville property for discrete harmonic functions is closely related to the recurrence property of electrical networks. Let $G_{\mu}=(V,E,\mu)$ be a weighted graph, i.e. a graph $G=(V,E)$ with a function $\mu:E\rightarrow [0,\infty)$.
A weighted graph $G_{\mu}$ could be viewed as an electrical network, where $\mu_{ij}$ denotes the conductance of the edge $ij$. Consider a function $u:V\rightarrow\mathbb{R}$, which denotes the electric potentials at the vertices. Then $u$ is harmonic at $i$ if and only if the outward electric flux at $i$ is $0$.
The theory of electrical networks is closely related to discrete extremal length, originally introduced by Duffin \cite{duffin1962extremal}. Here we briefly review the theory of discrete extremal length. All the definitions and properties here are well-known and one may read \cite{he1999rigidity} for references.

Let $V_1,V_2$ be two nonempty disjoint subsets of $V$. Denote $\Gamma(V_1,V_2)$ as the set of the paths joining $V_1$ and $V_2$. For $V_2=\infty$, we denote $\Gamma(V_1,\infty)$ as the set of the paths starting from $V_1$ and passing through infinite distinct vertices. An \emph{edge metric} is a function $m:V\rightarrow [0,\infty)$. An edge metric $m$ is called $\Gamma(V_1,V_2)$-admissible if $\sum\limits_{e\in \gamma}m(e)\geq 1$ for every $\gamma\in \Gamma(V_1,V_2)$. The \emph{resistance} between $V_1$ and $V_2$ is 
denoted as $\text{RES}(V_1,V_2)$ and
defined as
$$
\text{RES}(V_1,V_2)^{-1}=\inf\{\sum\limits_{e\in E}\mu(e)m^2(e):~m \text{ is $\Gamma(V_1,V_2)$-admissible}\}.
$$
An electrical network $G_{\mu}$ is called \emph{connected} if $G^{\prime}_{\mu}=(V,E\setminus\{e\in E:\mu(e)=0\})$ is connected.
Let $V_0$ be a nonempty finite subset of vertices in $G_{\mu}$. Then a connected electrical network $G_{\mu}$ is called \emph{recurrent} if $\text{Res}(V_0,\infty)=\infty$, and \emph{transient} otherwise. The recurrency of $G_\mu$ is independent of the choice of $V_0$. The following lemma is well-known and shows that the recurrence implies a discrete Liouville property. See Lemma 5.5 in \cite{he1999rigidity} for a proof.
\begin{lemma}\label{recurrent implies Liouville}
Suppose $G_{\mu}$ is recurrent and $u$ is a bounded function on $V$. If $u$ is harmonic on every point of $V$, then $u$ is constant.
\end{lemma}
To check that if an electrical network $G_{\mu}$ is recurrent, one may study the so-called vertex extremal length of the graph $G$.
Given a graph $G=(V,E)$, a \emph{vertex metric} is a function $\eta:V\rightarrow [0,\infty)$. A vertex metric $\eta$ is called $\Gamma(V_1,V_2)$-admissible if $\sum\limits_{v\in \gamma}\eta(v)\geq 1$ for every $\gamma\in \Gamma(V_1,V_2)$. The \emph{vertex extremal length} between $V_1$ and $V_2$ is denoted as $\text{VEL}(V_1,V_2)$ and defined as
$$
\text{VEL}(V_1,V_2)^{-1}=
\inf\{\sum\limits_{v\in V}\eta^2(v):~\eta \text{ is $\Gamma(V_1,V_2)$-admissible}\}
.
$$
Let $V_0$ be a nonempty finite subset of vertices in a connected graph $G$. Then $G$ is called \emph{$\text{VEL}$-parabolic} if $\text{VEL}(V_0,\infty)=\infty$, and \emph{$\text{VEL}$-hyperbolic} otherwise. The definition is independent of the choice of $V_0$.

The relation between the $\text{VEL}$-parabolicity and the recurrence is as follows.
\begin{lemma}[Lemma 5.4 in \cite{he1999rigidity}]\label{VEL-parabolic implies recurrent}
Let $C>0$ be a constant. Suppose that for each vertex $v$, we have $\sum\limits_{v\in e}\mu(e)\leq C$. Then for any mutually disjoint, nonempty subsets $V_1$ and $V_2$ of $V$, we have
$$
\text{VEL}(V_1,V_2)\leq 2C\cdot\text{Res}(V_1,V_2).
$$
In particular, if $G$ is VEL-parabolic and $G_{\mu}$ is connected, then $G_{\mu}$ is recurrent.
\end{lemma}

\section{Proof of the Main Theorem}\label{Proof of the Main Theorem}
We will prove our main Theorem \ref{main} by proving Propositions \ref{boundedness} and \ref{boundedness to rigidity}. We will first derive needed estimates on uniformly nondegenerate triangulations in Section \ref{Estimates on uniformly nondegenerate triangulations}. A discrete Schwarz lemma is developed in Section 4.2, for the proof of Proposition \ref{boundedness} in Section 4.3. A discrete Liouville theorem is developed in Section 4.4, for the proof of Proposition \ref{boundedness to rigidity} in Section 4.5.

\subsection{Estimates on uniformly nondegenerate triangulations}
\label{Estimates on uniformly nondegenerate triangulations}
\begin{lemma}\label{degree estimate}
Suppose $\phi:|T|\rightarrow\mathbb C$ is a geodesic embedding and all the inner angles in $l=l(\phi)$ are at least $\epsilon$. Then we have the following.

(a) For all $i\in V$, $\deg (i)\leq \frac{2\pi}{\epsilon}$.

(b) For all $\triangle ijk\in F$,
$$
\sin\epsilon\leq \frac{l_{ij}}{l_{ik}}\leq \frac{1}{\sin\epsilon}.
$$

(c) For all $\triangle ijk$,
$$
\frac{\sin^2\epsilon}{2}l_{ij}^2\leq\emph{Area}(\phi(\triangle ijk))\leq \frac{1}{2\sin\epsilon}l_{ij}^2.
$$

(d) There exists a constant $\delta=\delta(\epsilon)>0$ such that for all $\triangle ijk\in F$ with $i,j,k\in\emph{int}(V)$,
$$
d(U_{ijk}^c,\phi(\triangle ijk))\geq\delta \cdot \emph{diam}(\phi(\triangle ijk)),
$$
where
$$
U_{ijk}=\emph{int}(\phi(R_i))\cup \emph{int}(\phi(R_j))\cup \emph{int}(\phi(R_k)).
$$

(e) Suppose $a\in V$ and $\phi(a)=0$. Assume $r>0$ is such that
$$
\phi(R_a)\subseteq D_r\subseteq \phi(|T|).
$$
Denote $V_1=\{i\in V:\phi(i)\in D_r\}$ and $T_1$ as the subcomplex generated by $V_1$. Then there exists a constant $C=C(\epsilon)>0$ such that $D_{r/C}\subseteq\phi(|T_1|).$
\end{lemma}
\begin{proof}
(a) It is from
$$
2\pi\geq\sum\limits_{jk:\triangle ijk\in F}\theta^i_{jk}\geq \sum\limits_{jk:\triangle ijk\in F}\epsilon=\deg (i)\cdot \epsilon.
$$

(b) This is by the sine law.

(c) This estimate is straightforward from the area formula
$$
\text{Area}(\phi(\triangle ijk))=\frac{1}{2}l_{ij}l_{ik}\sin\theta^i_{jk}
=\frac{1}{2}l_{ij}^2\frac{\sin \theta^j_{ik}}{\sin\theta^k_{ij}}\sin\theta^i_{jk}.
$$

(d)
We may normalize $\diam(\phi(\triangle ijk))=1$ and $\phi(i)=0$.
By part (a), there are finitely many possible combinatorial structures of the natural triangulation of $R_i\cup R_j\cup R_k$.
Fixing a combinatorial structure,
$d(U_{ijk}^c,\phi(\triangle ijk))$
is positive and continuously determined by
$\phi(a)$'s for $a\in N_i\cup N_j\cup N_k$.
By the compactness $d(U_{ijk}^c,\phi(\triangle ijk))$ has a positive lower bound $\delta=\delta(\epsilon)$.

(e)
Pick $\delta$ as in part (d) and we claim that $C=1+2/\delta$ is a desired constant.
Let us prove by contradiction and assume
that there exists $z\in D_{r/C}\backslash\phi(|T_1|)$. Then there exists a triangle $\triangle ijk\in F$ such that $z\in\triangle ijk$. Then $\triangle ijk$ is not a triangle in $T_1$ and we may assume $i\notin V_1$, so $|\phi(i)|\geq r$. Since $\phi(R_a)\subseteq D_r$, we have $ai\notin E$. So $a\neq i,j,k$, then $0=\phi(a)\notin U_{ijk}$.
Then
\begin{eqnarray*}
&&\quad r/C\geq|0-z|
\geq d(U_{ijk}^c,\triangle ijk)
\geq\delta\cdot\diam(\triangle ijk)\\
&&\geq \delta\cdot |\phi(i)-z|
\geq\delta\cdot(r-r/C)
=(r/C)\cdot\delta(C-1)=2r/C
\end{eqnarray*}
and we get a contradiction.
\end{proof}

\subsection{A discrete Schwarz lemma}
Recall that the Schwarz lemma says that any holomorphic map $f:D\rightarrow D$ satisfies that $|f'(0)|\leq1$ if $f(0)=0$. Here we prove a discrete weaker version of the Schwarz lemma. Let $T$ be a triangulated surface.
\begin{prop}\label{key estimate}
Suppose $\phi,\phi'$ are geodesic embeddings of $|T|$ into $\mathbb{C}$ with induced PL metrics $l,l'$.
Assume both $l,l'$ satisfy the uniformly nondegenerate condition with constant $\epsilon>0$ and the Delaunay condition.
If $r,r'>0$ and $T_0$ is a finite subcomplex  of $T$ satisfing that
$$
\phi(|T_0|)\subseteq D_r,\quad D_{r'}\subseteq \phi'(|T_0|),
$$
then there exists a constant $M=M(\epsilon)>0$ such that for every $i\in V$ satisfying $\phi'(i)\in D_{r'/2}$, we have
$$u_i\geq \log\frac{r'}{r}-M.$$
\end{prop}

\begin{proof}
Without loss of generality, we assume $\epsilon\leq\pi/6$.
By scaling, we may assume $r=\frac{1}{4}\sin^3\epsilon\leq \frac{1}{4}$ and $r'=1$.
Denote
$$
V_1=\{i\in V:\phi'(i)\in D=D_1\}
$$
and $T_1=T(V_1)$ as the subcomplex generated by $V_1$. Then $\phi,\phi'$ map $|T_1|$ into $D$.
Let $z_i=\phi(i)$, $z_i'=\phi'(i)$.
Denote
$$
u^h_i=u_i+\ln\frac{1-|z_i|^2}{1-|z_i'|^2}.
$$

We claim $u_i^h\geq 0$ for every $i\in V_1$. Just let $i$ attain the minimum of $u^h$ in $V_1$. Here $\text{int}(V_1)$ and $\partial V_1$ are with respect to $T$, and defined in Section \ref{Notations and conventions}.\\
(1) If $i\in \text{int}(V_1)$ and $l'_{ij}<(1-|z'_{i}|^2)\sin\epsilon$, for every $ij\in E$, then from Lemma \ref{hyperbolic embedding}, $\phi'$ induces a hyperbolic geodesic embedding $\phi^{h\prime}$ from the $1$-ring neighborhood $R_{i}$ of $i_0$ into $\mathbb{D}$. Since $\phi'$ is Delaunay, $\phi^{h\prime}$ is also Delaunay. Since $\phi(|V_1|)\subset D_{\sin^3\epsilon/4}$, for the same reason $\phi$ also induces a Delaunay hyperbolic geodesic embedding $\phi^h$ from $R_{i}$ into $\mathbb{D}$. Then the hyperbolic maximum principle Lemma \ref{hyperbolic maximum principle} implies $u_{i}^h\geq 0$.\\
(2) If $i\in \text{int}(V_1)$ and there exists $j\in V_1$, $ij\in E$ such that $l'_{ij}\geq(1-|z'_{i}|^2)\sin\epsilon$, then from Lemma \ref{degree estimate} (b),
$$
e^{(u_i-u_j)/2}=\frac{e^{(u_i+u_k)/2}}{e^{(u_j+u_k)/2}}=\frac{l'_{ik}}{l_{ik}}\frac{l_{jk}}{l'_{jk}}\geq \sin^2\epsilon
$$
where $\triangle ijk\in F$. So
\begin{eqnarray*}
&&\quad e^{u_{i}^h}=e^{u_{i}}\cdot\frac{1-|z_{i}|^2}{1-|z_{i}'|^2}=\frac{l'_{ij}}{l_{ij}}\cdot e^{(u_{i}-u_{j})/2}\cdot \frac{1-|z_{i}|^2}{1-|z_{i}'|^2}\\
&&\geq\frac{l'_{ij}}{l_{ij}}\cdot\sin^2\epsilon\cdot\frac{1-|z_{i}|^2}{1-|z_{i}'|^2}
\geq\sin^3\epsilon\cdot\frac{1-|z_{i}|^2}{l_{ij}}
\geq\sin^3\epsilon\cdot\frac{1/2}{2r}=1.
\end{eqnarray*}
(3) If $i\in\partial V_1$, then there exists $j\in V$ such that $\phi'(j)\notin D$.
Then $l'_{ij}\geq1-|z'_{i}|.$ Since we assume $\epsilon\leq \frac{\pi}{6}$, then
$l'_{ij}\geq(1-|z'_{i}|^2)\sin\epsilon$.
As the estimates above, we also have $e^{u_{i}^h}\geq 1.$

So $u^h_i\geq 0$ for every $i\in V_1$. Then for $i\in V$ satisfying $\phi'(i)\in D_{1/2}$, set $M=-\ln\frac{\sin^3\epsilon}{8}$, and we have
$$
u_i=u_i^h-\ln\frac{1-|z_i|^2}{1-|z_i'|^2}\geq -\ln\frac{1-|z_i|^2}{1-|z_i'|^2}\geq \ln (1-|z_i'|^2)\geq\ln\frac{1}{2}=\ln\frac{r'}{r}-M.
$$
\end{proof}
\subsection{Proof of the boundedness of the conformal factor}
In this section, we prove that the discrete conformal factor $u$ is bounded, i.e. Proposition \ref{boundedness}.
\begin{proof}[Proof of Proposition \ref{boundedness}]
Without loss of generality, we may assume that $\phi'\circ\phi^{-1}$ is linear on each triangle $\phi(\triangle ijk)$. Then $\phi'\circ\phi^{-1}$ is $K$-quasiconformal for some constant $K=K(\epsilon)>0$.
We will prove the boundedness of $u$ by showing that for every $j_1,j_2\in V$,
$$
|u_{j_1}-u_{j_2}|\leq 2M+2\ln C+\ln C'-\ln2,
$$
where $M=M(\epsilon)$ is the constant given in Proposition \ref{key estimate} and $C=C(\epsilon)$ is the constant given in Lemma \ref{degree estimate} (e) and $C'=C'(\epsilon)=e^{200\pi K}$.

Assume $j_1,j_2\in V$. For convenience, let us assume $\phi(j_1)=\phi'(j_1)=0$ by translations.
Pick $r>0$ sufficiently large such that $|\phi(j_2)|<r/(2C)$ and $\phi(R_{j_1})\subseteq D_{r}$.
Let
$$
V_1=\{i\in V:\phi(i)\in D_{r}\}\text{ and }V_2=\{i\in V:\phi(i)\in D_{CC'r}\}.
$$
Denote $T_1,T_2$ as the subcomplexes generated by $V_1,V_2$ respectively.
Then by Lemma \ref{degree estimate} (e) we have
    \begin{equation}
    \label{31}
        \{\phi(j_1),\phi(j_2)\}\subseteq D_{r/(2C)}\subseteq D_{r/C}\subseteq\phi(|T_1|)\subseteq D_r.
    \end{equation}
And
    \begin{equation}
    \label{32}
       D_{C'r}\subseteq \phi(|T_2|)\subseteq D_{CC'r}.
    \end{equation}
Recall that $A_{r_1,r_2}=\{z\in\mathbb{C}:r_1<|z|<r_2\}$. Then $A=A_{r,C'r}$ separates $\phi(|T_1|)$ and $\phi(|T_2|)^c$, and then $A'=\phi'\circ\phi^{-1}(A)$ separates
$\phi'(T_1)$ and $\phi'(T_2)^c$.
Furthermore
$$
\text{Mod}(A')\geq\frac{1}{K}\cdot\text{Mod}(A)=\frac{1}{K}\cdot\frac{1}{2\pi}\ln\frac{C'r}{r}=100.
$$
Then by Lemma \ref{modulus} there exists $r'>0$ such that
$A_{r',2r'}\subseteq A'$. So $A_{r',2r'}$ separates $\phi'(T_1)$ and $\phi'(T_2)^c$ and then
\begin{equation}
\label{33}
\phi'(|T_1|)\subseteq D_{r'}
\end{equation}
and
\begin{equation}
\label{34}
\{\phi'(j_1),\phi'(j_2)\}
\subseteq D_{r'}\subseteq D_{2r'}
\subseteq\phi'(|T_2|).
\end{equation}
By Proposition \ref{key estimate}, Equation (\ref{32}) and Equation (\ref{34}), both $u_{j_1},u_{j_2}$ are at least
$$
\ln\frac{2r'}{CC'r}-M
=\ln\frac{r'}{r}+\ln\frac{2}{CC'}-M.
$$
Again by Proposition \ref{key estimate}, Equation (\ref{33}) and Equation (\ref{31}),
both $-u_{j_1}$ and $-u_{j_2}$ are at least
$$
\ln\frac{r/C}{r'}-M=\ln\frac{r}{r'}-\ln C-M.
$$
So both $u_{j_1}$ and $u_{j_2}$ are in the interval
$$
[
\ln\frac{r'}{r}+\ln\frac{2}{CC'}-M
,\ln\frac{r'}{r}+\ln C+M
],
$$
and $|u_{j_1}-u_{j_2}|$ is bounded by the length of this interval
$$
2M+\ln C-\ln\frac{2}{CC'}=2M+2\ln C+\ln C'-\ln 2.
$$
\end{proof}

\subsection{A discrete Liouville theorem}
In this section, we prove the following discrete Liouville theorem.
\begin{prop}\label{recurrence}
Suppose $\phi:|T|\rightarrow \mathbb{C}$ is a geodesic homeomorphism, and $l$ satisfies the uniformly nondegenerate condition and the Delaunay condition. Given the weight $\mu\in\mathbb R_{\geq0}^E$ defined as in equation (\ref{weight mu}), then any bounded function $u$ on $V$ is constant if $u$ is harmonic at every point of $V$.
\end{prop}


\begin{proof}
By Lemma \ref{recurrent implies Liouville}, it suffices to show $(V,E,\mu)$ is recurrent.
Let us assume $l$ is uniformly nondegerate with constant $\epsilon>0$.
Then for all $i\in V$,
$$
\sum_{j:ij\in E}\mu_{ij}\leq\deg(i)\cdot\cot\epsilon\leq\frac{2\pi\cot\epsilon}{\epsilon}.
$$
Then by Lemma \ref{VEL-parabolic implies recurrent}, it suffices to show $(V,E)$ is VEL-parabolic and $(V,E,\mu)$ is connected.

First we show that if two subsets $V_1,V_2$ are separated by an annulus, then the vertex extremal length has a lower bound.
\begin{lemma}\label{estimate of VEL containing annulus}
Let $V_1,V_2$ be two nonempty subsets of $V$ such that $\phi(V_1)\subseteq D_{r_1}$, $\phi(V_2)\subseteq D^c_{r_2}$. Suppose for every $i\in V_1$, $\phi(|R_i|)\subseteq D_{r_2}$. Then there is a constant $C=C(\epsilon)>0$ such that 
$$
\text{VEL}(V_1,V_2)\geq \frac{1 - (r_1/r_2)^2}{C}.
$$
In particular, if $r_2\geq 2r_1$ then $\text{VEL}(V_1,V_2)\geq \frac{1}{2C}.$
\end{lemma}
\begin{proof}
Denote
$d_{M}(i)=\max\limits_{j:ji\in E}d(\phi(i),\phi(j))$.
Consider the vertex metric $\eta$ as follows.
For $i\in V$, if $\phi(i)\in D_{r_2}^c$, then define $\eta(i)=0$; if $\phi(i)\in D_{r_2}$ and there exists $j\in V$, $ji\in E$ such that $\phi(j)\in  D_{r_2}$, then define $\eta(i)=\frac{d_M(i)}{r_2-r_1}$; if $\phi(i)\in D_{r_2}$ and there is no $j\in V$, $ji\in E$ such that $\phi(j)\in  D_{r_2}$, then define $\eta(v)=0$.

First we check $\eta$ is $\Gamma(V_1,V_2)$-admissible. Let $\gamma=\{i_0,...,i_k\}$ be a path joining $V_1$ and $V_2$ such that $i_0,\cdots,i_{k-1}\in D_{r_2}$ and $i_{k}\in D_{r_2}^c$. Then from the assumptions
\begin{eqnarray*}
&&\sum\limits_{s=0}^{k-1}\eta(i_{s})=\frac{1}{r_2-r_1}\sum\limits_{s=0}^{k-1}d_M(i_s)\\
&&\geq\frac{1}{r_2-r_1}\sum\limits_{s=0}^{k-1}d(\phi(i_s),\phi(i_{s+1}))\geq\frac{1}{r_2-r_1}d(\phi(i_0),\phi(i_{k+1})).
\end{eqnarray*}
Since $\phi(i_0)\in V_1\subseteq D_{r_1}$ and $\phi(i_{k+1})\in D_{r_2}^c$, we obtain $\sum\limits_{i\in \gamma}\eta(i)\geq \frac{1}{r_2-r_1}(r_2-r_1)=1.$

Next we estimate an upper bound of $\sum\limits_{i\in V}\eta^2(i)$.
We only need to consider the vertices where $\eta$ are nonzero. For $i\in V$, $\eta(i)\neq 0$, since $\phi(i)\in D_{r_2}$ and there is $j\in V$ such that $\phi(j)\in D_{r_2}$ and $ji\in E$, we have $d(\phi(i),\phi(j))\leq 2r_2$.
Then from Lemma \ref{degree estimate} (a) and (b), there is a constant $C_1=C_1(\epsilon)>0$ such that $d_M(i)\leq C_1r_2$. So for every $i\in V$, $\eta(i)\neq 0$, we have $\phi(R_i)\subset D_{(1+C_1)r_2}.$
Then from Lemma \ref{degree estimate} (a) and (b) and (c), there is a constant $C_2=C_2(\epsilon)>0$ such that
$$
\sum\limits_{i\in V}\eta^2(i)=\sum\limits_{\eta(i)\neq 0}\frac{d_M^2(i)}{(r_2-r_1)^2}\leq \sum\limits_{\eta(i)\neq 0}\frac{C_2\text{Area}(\phi(R_i))}{(r_2-r_1)^2}.
$$
Since every triangle is calculated at most three times in $\phi(|R_i|)$, $i\in V$, we obtain
$$
\sum\limits_{i\in V}\eta^2(i)\leq \frac{3C_2\text{Area}(D_{(1+C_1)r_2})}{(r_2-r_1)^2}=\frac{Cr_2^2}{(r_2-r_1)^2}.
$$
where
$$
C=C(\epsilon)=3C_2\cdot\pi(1+C_1)^2>0.
$$
So 
$$
\text{VEF}(V_1,V_2)^{-1}\leq\frac{Cr_2^2}{(r_2-r_1)^2}
=\frac{C}{1-(r_1/r_2)^2}
$$ 
and we are done.
\end{proof}

The vertex extremal length has a property of additivity, which is from Lemma 5.1 in \cite{he1999rigidity}.
\begin{lemma}[Lemma 5.1 in \cite{he1999rigidity}]\label{sum separate annulus}
Let $V_1, V_2, \cdots, V_{2m}$ be mutually disjoint, nonvoid subsets of vertices such that for $i_1<i_2<i_3$, $V_{i_2}$ separates $V_{i_1}$ from $V_{i_3}$, i.e. every path joining $V_{i_1}$ and $V_{i_3}$ must pass through a vertex in $V_{i_2}$. We allow $V_{2m}=\infty.$ Then we have $$\text{VEL}(V_1,V_{2m})\geq \sum\limits_{i=1}^{m}\text{VEL}(V_{2k-1},V_{2k}).$$
\end{lemma}

Then combining Lemma \ref{estimate of VEL containing annulus} and Lemma \ref{sum separate annulus}, we can show the following lemma. \begin{lemma}\label{parabolic}
$(V,E)$ is VEL-parabolic.
\end{lemma}
\begin{proof}
We construct infinitely many $V_k$ by induction starting from a finite subset $V_1$ of $V$. Given a finite subset $V_k$ of $V$, choose $r$ large enough such that $\bigcup\limits_{i\in V_{k}}\phi(R_i)\subseteq D_r$. Denote $\tilde{V}=\{i\in V:\phi(i)\in D_{2r}\}$. Consider $V'=\bigcup\limits_{i\in \tilde{V}}R_i$, $V''=\bigcup\limits_{i\in V'}R_i.$ Choose $r'$ such that $\phi(V'')\subseteq D_{r'}$. Since the interior of $\phi(R_i)$, $i\in V$ cover the plane, we can choose a finite subset $V'_{k+1}$ such that $\partial D_{r'}\subseteq \bigcup\limits_{i\in V'_{k+1}}\text{int}(\phi(R_i))$, and we may assume $\partial D_{r'}\cap \text{int}(\phi(R_i))\neq \emptyset$. Then set
$$
V_{k+1}=\{i\in V:\text{ there exists }j\in V'_{k+1} \text{ such that }ji\in E\}.$$
We claim $\phi(V_{k+1})\subseteq D^c_{2r}$. If not, then there exists $i\in V'_{k+1}$ such that $i\in V'$. Then $\phi(R_i)\subseteq D_{r'}$, which does not intersect with $\partial D_{r'}$. Contradiction.

As construction above, we obtain that for $i\in V_k$, $\phi(R_i)\subseteq D_{r}$ and $\phi(V_{k+1})\subseteq D^c_{2r}$. Since $\partial D_{r'}\subseteq \bigcup\limits_{i\in V'_{k+1}}\phi(R_i)$, it is clearly $V_{k+1}$ separates $V_k$ and $V_{k+2}$. So the conditions in Lemma \ref{estimate of VEL containing annulus} and Lemma \ref{sum separate annulus} hold.
\end{proof}

It remains to show that $G_{\mu}=(V,E,\mu)$ is connected, i.e., $G^{\prime}_{\mu}$ is connected. We notice that if $\mu_{ij}=0$ then the points $\phi(i),\phi(k_1),\phi(j),\phi(k_2)$ are co-circle. So to obtain $G^{\prime}_{\mu}$, we just delete, for finitely many times, the interior edges of a polygon, whose vertices are co-circle. So $G^{\prime}_{\mu}$ is connected.
\end{proof}

\subsection{Proof of that the conformal factor is constant}
In this section, we show that $u$ is constant, i.e. Proposition \ref{boundedness to rigidity}, then finish the whole proof of main Theorem \ref{main}.
\begin{proof}[Proof of Proposition \ref{boundedness to rigidity}]
In the proof, we will construct a deformation of discrete conformal factors $u(t)$. To avoid being confused with $u$ and $u(t)$, we use the notation $\bar{u}$ instead of $u$ in the statement of Proposition \ref{boundedness to rigidity}. In other words, we assume $l'=\bar{u}*l$ and want to prove that $\bar{u}$ is constant.
Since $\bar{u}$ is bounded, $|\bar{u}|_{\infty}$ is finite. Assume $|\bar{u}|_{\infty}\neq 0$, otherwise $\bar{u}$ is constant.

By a standard compactness arguement, there exists a constant $\delta_0=\delta_0(\epsilon)$, such that for every function $u:V\rightarrow \mathbb{R}$ satisfying $|u|_\infty<2\delta_0$, $u*l$ satisfies the triangle inequalities and is uniformly nondegenerate and uniformly Delaunay.

Denote $\delta=\min\{\delta_0,|\bar u|_\infty\}$.
Pick a sequence of increasing finite subsets $V_n$ of $V$ such that  $\bigcup\limits_{n=1}^\infty V_n=V$. We use the notation in Section \ref{Discrete Harmonic Functions}.
For each $n\in\mathbb N$, we will construct a smooth $\mathbb R^{V_n}$-valued function $u^{(n)}(t)=[u_i^{(n)}(t)]_{i\in V_n}$ on $(-2\delta,2\delta)$ such that
\begin{enumerate}[label=(\alph*)]
    \item $u^{(n)}(0)=0$, and
    \item  $\dot u_i^{(n)}(t)=\bar u_i/|\bar u|_\infty$ if $i\in \partial V_n$, and
    \item  if $i\in \text{int}(V_n)$ then
    \begin{equation}
    \label{condition c}
    \sum_{j:ij\in E}\mu_{ij}(u^{(n)}(t))
(\dot u_j^{(n)}(t)
-\dot u_i^{(n)}(t))
=0
\end{equation}
\end{enumerate}
where $\mu_{ij}(u)$ is defined for all $ij\in E(V_n)$ as in Equation (\ref{weight mu}). Here Equation (\ref{condition c}) just means $\frac{d}{dt}K(u^{(n)}(t))=0$ by Equation (\ref{deformation of curvature}).

The conditions (b) and (c) give an autonomous ODE system on
$$
\mathcal U_n=\{u\in\mathbb R^{V_n}:|u|_\infty<2\delta\}.
$$
Notice that
$\mu_{ij}(u)>0$ if $u\in\mathcal U_n$.
Then by Lemma \ref{Dirichlet},
$\dot u^{(n)}(t)$ is smoothly determined by $u^{(n)}(t)$ on $\mathcal U_n$. Given the initial condition $u^{(n)}(0)=0$,
assume the maximum existence interval for this ODE system on $\mathcal U_n$ is $(t_{\min},t_{\max})$ where
$t_{\min}\in[-\infty,0)$ and $t_{\max}\in(0,\infty]$.
By the maximum principle Lemma \ref{discrete Laplacian maximum principle}, for all $i\in V_n$
$$
|\dot u^{(n)}|_\infty\leq \max_{j\in\partial V_n}|\dot u_j^{(n)}|=
\max_{j\in\partial V_n}|\bar u_j|/|\bar u|_\infty\leq 1.
$$
So $|u^{(n)}(t)|_\infty\leq t\leq t_{\max}$ for all $t\in [0,t_{\max})$. By the maximality of $t_{\max}$, $t_{\max}\geq2\delta$ and for a similar reason $t_{\min}\leq-2\delta$. So $u^{(n)}(t)$ is indeed well-defined on $(-2\delta,2\delta)$.
By Equation (\ref{deformation of curvature}) and Equation (\ref{condition c}), $K_i(u^{(n)}(t))=0$ for all $i\in \text{int}(V_n)$.
Since $K_i(\bar{u})=0$ and $K_i(u^{(n)}(\delta))=0$ and $\bar u-u^{(n)}(\delta)$ is the related conformal factor, then by Lemma \ref{mp}, for all $i\in V_n$
\begin{eqnarray}
\label{eqn32}
&&\quad |\bar u_i-u_i^{(n)}(\delta)|\leq \max_{j\in\partial V_n}|\bar u_j-u_j^{(n)}(\delta)|\notag\\
&&=\max_{j\in\partial V_n}
\left(
\bar u_j-\delta\cdot
\frac{\bar u_j}{|\bar u|_\infty}
\right)
\leq(1-\frac{\delta}{|\bar u|_\infty})|\bar u|_\infty
=|\bar u|_\infty-\delta.
\end{eqnarray}

By picking a subsequence, we may assume that $u^{(n)}_i$ converge to $u_i^*$ on $[0,\delta]$ uniformly for all $i\in V$.
Then $u^*=[u_i^{*}]_{i\in V}$ satisfies the following properties.

(a) $u^*_i(t)$ is $1$-Lipschitz for all $i\in V$. As a consequence, for all $i\in V$,
$u_i^*(t)$ is differentiable at a.e. $t\in[0,\delta]$.

(b) $u^*(t) * l$ is uniformly nondegenerate and uniformly Delaunay for all $t\in[0,\delta]$.

(c) For all $i\in V$, $K_i(u^*(t))=0$. As a consequence for a.e. $t\in[0,\delta]$,
$$
0=\frac{d}{dt}K_i(u^*(t))=\sum_{j:ij\in E}\mu_{ij}(u^*(t))(\dot u^*_j(t)-\dot u^*_i(t)),
$$
for all $i\in V$.

(d) By Proposition \ref{recurrence}, $\dot u^*(t)$ is constant on $V$ for a.e. $t\in[0,\delta]$. As a consequence $u_i^*(\delta)$ equals to a constant $c$ independent on $i\in V$.

(e) By Equation (\ref{eqn32}),
$$
|\bar u_i-u^*_i(\delta)|\leq|\bar u|_\infty-\delta
$$
for all $i\in V$. 

So the above constructed flow $u^*(t)$ indeed gives a similar transformation on the PL metric, and the $L^\infty$-distance from $u^*(t)$ to $\bar u$ decreases from
$|\bar{u}|_{\infty}$ to 
$$
|u|_\infty-\delta = \max\{0,|u|_\infty-\delta_0\}.
$$ 
We can continue to deform, or scale, the PL metric by repeatedly implementing a similar flow with the same constant $\delta_0=\delta_0(\epsilon)$. For at most $\lceil|u|_\infty/\delta_0\rceil$ times, we achieve the target discrete conformal factor $\bar u$ as a constant function on $V$.


\end{proof}

\bibliography{rigidity}

\begin{thebibliography}{DGM22}

\bibitem[Ahl10]{ahlfors2010conformal}
Lars~Valerian Ahlfors.
\newblock {\em Conformal invariants: topics in geometric function theory},
  volume 371.
\newblock American Mathematical Soc., 2010.

\bibitem[And05]{anderson2005hyperbolic}
James~W Anderson.
\newblock {\em Hyperbolic geometry}.
\newblock Springer, 2005.

\bibitem[BPS15]{bobenko2015discrete}
Alexander~I Bobenko, Ulrich Pinkall, and Boris~A Springborn.
\newblock Discrete conformal maps and ideal hyperbolic polyhedra.
\newblock {\em Geometry \& Topology}, 19(4):2155--2215, 2015.

\bibitem[DGM22]{dai2022rigidity}
Song Dai, Huabin Ge, and Shiguang Ma.
\newblock Rigidity of the hexagonal delaunay triangulated plane.
\newblock {\em Peking Mathematical Journal}, 5(1):1--20, 2022.

\bibitem[Duf62]{duffin1962extremal}
RJ~Duffin.
\newblock The extremal length of a network.
\newblock {\em Journal of Mathematical Analysis and Applications},
  5(2):200--215, 1962.

\bibitem[He99]{he1999rigidity}
Zheng-Xu He.
\newblock Rigidity of infinite disk patterns.
\newblock {\em Annals of Mathematics}, pages 1--33, 1999.

\bibitem[LSW22]{luo2022discrete}
Feng Luo, Jian Sun, and Tianqi Wu.
\newblock Discrete conformal geometry of polyhedral surfaces and its
  convergence.
\newblock {\em Geometry \& Topology}, 26(3):937--987, 2022.

\bibitem[Luo04]{luo2004combinatorial}
Feng Luo.
\newblock Combinatorial yamabe flow on surfaces.
\newblock {\em Communications in Contemporary Mathematics}, 6(05):765--780,
  2004.

\bibitem[LV73]{lehto1973quasiconformal}
Olli Lehto and Kaarlo~Ilmari Virtanen.
\newblock {\em Quasiconformal mappings in the plane}, volume 126.
\newblock Citeseer, 1973.

\bibitem[WGS15]{wu2015rigidity}
Tianqi Wu, Xianfeng Gu, and Jian Sun.
\newblock Rigidity of infinite hexagonal triangulation of the plane.
\newblock {\em Transactions of the American Mathematical Society},
  367(9):6539--6555, 2015.

\end{thebibliography}
\bibliographystyle{alpha}

\end{document}